\documentclass[12pt]{amsart}
\usepackage{amsmath}
\usepackage{amstext}
\usepackage{amsfonts}
\usepackage{amssymb}
\usepackage{amsthm}
\usepackage[all]{xy}
\usepackage{comment}
\calclayout

\usepackage{microtype}
\usepackage{color,hyperref}
\definecolor{darkblue}{rgb}{0.0,0.0,0.3}
\hypersetup{colorlinks,breaklinks,
            linkcolor=red,urlcolor=red,
            anchorcolor=red,citecolor=red}

\theoremstyle{plain}
\newtheorem{thm}{Theorem}[section]

\newtheorem{cor}[thm]{Corollary}
\newtheorem{prop}[thm]{Proposition}
\newtheorem{lem}[thm]{Lemma}

\theoremstyle{definition}
\newtheorem{defn}[thm]{Definition}
\newtheorem{example}[thm]{Example}
\newtheorem{rem}[thm]{Remark}
\newtheorem{question}[thm]{Question}

\newtheorem*{thmA}{Theorem A}

\numberwithin{equation}{section}

\newcommand{\fix}{\operatorname{Fix}}
\newcommand{\tdel}{\widetilde{\delta}}
\newcommand{\re}{\operatorname{Re}}

\newcommand{\matsp}{\operatorname{mat-span}}
\newcommand{\Max}{\operatorname{Max}}

\def\mcc{M\raise.5ex\hbox{c}C}
\def\mccarthy{M\raise.5ex\hbox{c}Carthy}

\newcommand{\cF}{{\mathcal F}}

\newcommand{\cH}{{\mathcal H}}
\newcommand{\cI}{{\mathcal I}}

\newcommand{\cK}{{\mathcal K}}
\newcommand{\cL}{{\mathcal L}}

\newcommand{\cO}{{\mathcal O}}

\newcommand{\cT}{{\mathcal T}}

\newcommand{\fB}{{\mathfrak B}}

\newcommand{\fS}{{\mathfrak S}}

\newcommand{\fV}{{\mathfrak V}}
\newcommand{\fW}{{\mathfrak W}}

\newcommand{\C}{{\mathbb C}}
\newcommand{\R}{{\mathbb R}}

\newcommand{\N}{{\mathbb N}}

\newcommand{\M}{{\mathbb M}}

\newcommand{\D}{{\mathbb D}}

\newcommand{\BB}{{\mathbb B}}
\newcommand{\HH}{{\mathbb H}}

\newcommand{\bA}{\ensuremath{{\mathbf A}}}

\newcommand{\bX}{\ensuremath{{\mathbf X}}}
\newcommand{\bY}{\ensuremath{{\mathbf Y}}}
\newcommand{\bZ}{\ensuremath{{\mathbf Z}}}
\newcommand{\bW}{\ensuremath{{\mathbf W}}}

\usepackage{mathrsfs}

\newcommand{\GL}{\operatorname{{\mathbf GL}}}

\begin{document}
\title{Iteration theory of noncommutative maps}

\begin{abstract}
This note aims to study the iteration theory of noncommutative self-maps of bounded matrix convex domains. We prove a version of the Denjoy-Wolff theorem for the row ball and the maximal quantization of the unit ball of $\C^d$. For more general bounded matrix convex sets, we prove a version of Wolff's theorem inspired by the results of Abate. Lastly, we use iteration and fixed point theory to generalize the commutative results of Davidson, Ramsey, and Shalit to quotients of the free semigroup algebra by WOT closed ideals. 
\end{abstract}

\author{Serban T. Belinschi}
\address{CNRS - IMT}
\thanks{This work was initiated in November 2019 while STB was a visiting fellow supported by the Faculty of Natural Sciences Distinguished Scientist Visitors Program of the Ben Gurion University of the Negev, Beer-Sheva, Israel.}

\author{Eli Shamovich}
\email{shamovic@bgu.ac.il}
\address{Ben-Gurion University of the Negev}

\maketitle


\section{Introduction}

This note aims to study the iteration theory of noncommutative functions on matrix convex sets and, more generally, noncommutative domains. Noncommutative analysis is a rapidly growing area of research. The origins of noncommutative analysis can be traced back to the works of Takesaki \cite{Takesaki}, Taylor \cite{Taylor-frame, Taylor-ncfunc}, and Voiculescu \cite{Voic-coalgebra, Voic04, Voic10}. Though their motivations varied greatly, the notion of a noncommutative (nc for short) function is apparent in all of these works. Additional impetus to the field was given by results in noncommutative convexity with a view towards applications (see, for example, \cite{DHM17, Hel02, HMV06}). Arveson \cite{Arveson-I} pioneered a different approach to noncommutative convexity. Arveson has generalized the classical notion of Choquet boundary to the setting of operator systems. In particular, Arveson's extension theorem led Wittstock to prove an extension theorem on completely bounded maps and introduce matrix convex sets \cite{Wittstock-HB, Wittstock-survey}. The reader should consult the monograph \cite{DavKen-ncchoquet} and the references therein for more information on noncommutative convexity. A basic introduction to nc domains and functions is given in Section \ref{sec:ncfunc}. We refer the reader to the books \cite{AgMcCYng-book, KVV} for more details on nc function theory.

Classically, open convex domains in $\C^d$ are a primary example of hyperbolic spaces in the sense of Kobayashi \cite{Kobayashi-book}. A fundamental result of Lempert \cite{Lempert-metrique, Lempert-intrinsic} states that the Caratheodory and Kobayashi metrics on a convex bounded domain coincide (see also \cite{Agler-lempert} and \cite{RoydenWong}). This hyperbolic metric is an essential tool in the study of iteration theory and fixed points of holomorphic self-maps of convex domains \cite{Abate-horospheres, Abate-book, AbateRaissy, Bedford, Vigue-geodesique, Vigue-tams}. Nc hyperbolic metrics were introduced by the first author and Vinnikov in \cite{BelVin17}. The definition and basic properties of the nc
Lempert function and hyperbolic metrics appear in Section \ref{sec:hup_metric}.

A particularly nice example of a convex set is the unit ball; in the noncommutative case, it is the row ball. Popescu has extensively studied nc function theory on the row ball (see, for example, \cite{Popescu-funcI, Popescu-hyp_geom, Popescu-aut, Popescu-funcII, Popescu-composition}). Davidson and Pitts \cite{DavPit-inv, DavPit-alg, DavPit-NPint} have studied the free semigroup algebra and, in doing so, studied nc bounded analytic functions on the row ball (see also \cite{3S1,3S2}). In particular, Popescu proved a noncommutative version of Wolff's theorem for the row ball \cite[Theorem 3.1]{Popescu-composition}. A noncommutative version of a fixed point theorem of Rudin \cite{Rudin-ball_book} and Herv\'e \cite{Herve} was obtained by the second author in \cite{Sha18}. The primary motivation for the proof of the latter fixed point theorem was the classification of quotients of the free semigroup algebra by WOT closed ideals up to completely isometric isomorphism \cite{3S1,3S2}. 

In Section \ref{sec:denjoy_wolff}, we prove a version of the Denjoy-Wolff theorem for the row ball (Theorem \ref{thm:nc_denjoy_wolff_ball} and the maximal quantization of the unit ball $\BB_d \subset \C^d$ (Corollary \ref{cor:nc_denjoy_wolff_max}). In particular, this answers a question raised by Popescu in \cite{Popescu-composition} on the Denjoy-Wolff theorem for the row ball. One of the key ingredients of the proof of these results are forms of the nc maximum modulus principle. Popescu obtained one version of the maximum modulus principle for the row ball \cite[Theorem 3.3]{Popescu-funcI}, and a different version (more applicable to the problem at hand) was obtained in \cite[Lemma 6.11]{3S1}. In this paper, we obtain a maximum modulus principle for the maximal quantization of $\BB_d$ (see Theorem \ref{thm:max_modulus_max}). It is interesting whether such a strong form of maximum modulus principle is available for every quantization of the unit ball.

Section \ref{sec:nc_vigue} contains an nc version of the classical Vigu\'{e} and Bedford theorems on iterations of holomorphic functions (Theorems \ref{thm:nc_vigue} and \ref{thm:nc_bedford}). We use these results to obtain an extension of \cite[Theorem 4.1]{Sha18} that is the full nc version of a result of Davidson, Ramsey, and Shalit \cite[Theorem 4.4]{DRS15}.
\begin{thmA}[Theorem \ref{thm:nc_DRS}]
Let $\fV \subset \fB_d$ and $\fW \subset \fB_e$ be two nc analytic varieties. If $H^{\infty}(\fV)$ and $H^{\infty}(\fW)$ are completely isometrically isomorphic, then there exists $k \in \N$ and an automorphism $F\colon \fB_k \to \fB_k$, such that $\fV, \fW \subset \fB_k$ and $F$ maps $\fV$ onto $\fW$.
\end{thmA}

Lastly, in Section \ref{sec:horospheres}, we provide an nc version of Wolff's theorem in the spirit of the result of Abate \cite{Abate-horospheres}. Namely, we define nc analogs of big and small horospheres centered at a scalar point on the boundary of a bounded matrix convex set and prove that iterations of an nc analytic self-map send the small one into the big one.

\section{Noncommutative functions and domains} \label{sec:ncfunc}

Let $E$ be a vector space, we set $\M(E) = \bigsqcup_{n=1}^{\infty} M_n(E)$, where $M_n(E) = M_n \otimes E$. We write $\M_d = \M(\C^d)$ and treat $\M_d$ as the set of all $d$-tuples of matrices of all sizes. Note that $\GL_n$ acts on $M_n(E)$ by $S \cdot (A 
\otimes \xi) = (S^{-1} A S) \otimes \xi$. We will denote this action for $\bX \in M_n(E)$ simply by $S^{-1}\bX S$.

Additionally, for $\bX \in M_n(E)$ and $\bY \in M_m(E)$ we will write \[
\bX \oplus \bY = \begin{pmatrix} \bX & 0 \\ 0 & \bY \end{pmatrix} \in M_{n+m}(E).
\]
A subset $\Omega\subset \M(E)$ is called an nc set, if for every $\bX, \bY \in \Omega$, $\bX \oplus \bY \in \Omega$.

If $E$ is an operator space, then we set:
\[
\fB_E = \bigsqcup_{n=1}^{\infty} \left\{\bX \in M_n(E) \mid \|\bX\|_n < 1\right\}.
\]
We endow $\C^d$ with the row norm and write $\fB_d$ for $\fB_{\C^d}$. We will write $\BB_d = \fB_d(1)$ for the unit ball of $\C^d$ and $\D$ for the disc. A subset $\Omega \subset \M(E)$ will be called an nc domain if $\Omega$ is an nc set, level-wise open, and invariant under unitary similarities.

An nc set $\Omega \subset \M(E)$ is called matrix convex if, for every $\bX \in \Omega(n)$ and $V \colon \C^m \to \C^n$ an isometry, we have that $V^* \bX V \in \Omega(m)$.

Let $\Omega \subset \M(E)$ be an nc domain. We say that a function $f \colon \Omega \to \M(F)$ is an nc function, if
\begin{itemize}
    \item $f$ is graded, i.e, $f(\Omega(n)) \subset M_n(F)$, for every $n \in \N$. We write $f_n = f|_{\Omega(n)}$.
    
    \item $f$ respects direct sums, i.e, $f(\bX \oplus \bY)= f(\bX) \oplus f(\bY)$.
    
    \item $f$ respects similarities, for every $n \in \N$, every $\bX \in \Omega(n)$, and every $S \in \GL_n$, if $S^{-1} \bX S \in \Omega(n)$, then $f(S^{-1} \bX S) = S^{-1} f(\bX) S$.
\end{itemize}
It is a remarkable result of Kalyuzhnyi-Verbovetskyi and Vinnikov that a locally bounded nc function is both Gateaux and Frechet differentiable (see also \cite{AgMcC15-GHF}). Hence, we will say an analytic nc function when we mean locally bounded. The key to proving this result is the nc difference-differential operator. Namely, for every $\bX \in \Omega(n)$, $\bY \in \Omega(m)$, there exists a linear map $\Delta f(\bX, \bY) \colon M_{n,m} \otimes E \to M_{n,m} \otimes E$, such that for every $\bZ \in M_{n,m} \otimes E$, such that $\left(\begin{smallmatrix} \bX & \bZ \\ 0 & \bY \end{smallmatrix}\right) \in \Omega(n+m)$,
\[
f \left( \begin{pmatrix} \bX & \bZ \\ 0 & \bY \end{pmatrix} \right) = \begin{pmatrix} f(\bX) & \Delta f(\bX,\bY)(\bZ) \\ 0 & f(\bY) \end{pmatrix}.
\]
In fact, $\Delta f(\bX, \bX)$ is the Frechet derivative of a locally bounded nc function at $\bX$. Moreover, we have the identity
\[
\Delta f(\bX,\bY)(\bX - \bY) = f(\bX) - f(\bY).
\]
For the properties of the nc difference-differential operator, we refer the reader to \cite{KVV}.

The nc sets that we will focus on are described in the following definition.
\begin{defn} \label{def:D_Q}
Let $\cH$, $\cK$ be a Hilbert space. Let $\Omega \subset \M(E)$ be an nc domain and $Q \colon \Omega \to \M(B(\cH,\cK))$ be an nc analytic function. We define $\D_{Q} = \left\{\bZ \in \Omega \mid Q(\bZ) Q(\bZ)^* < I \right\}$. 
\end{defn}

\begin{rem}
For $\bX \in \Omega(n)$, $\bY \in \Omega(m)$, and $\bZ \in M_{n,m} \otimes B(\cH)$, define the nc affine kernel $G(\bX,\bY)(\bZ) = \bZ - Q(\bX) \bZ Q(\bY^*)^*$.  Then, $\D_Q$ coincides with a generalized noncommutative ball, as defined in \cite{BelVin17}. Domains of the form $\D_Q$ were originally defined by Agler and McCarthy for $Q$, a matrix of nc polynomials (see, for example, \cite{AgMcC15-GHF, AgMcC16}). In this form, the sets $\D_Q$ first appeared in the works of Ball, Marx, and Vinnikov \cite{BMV-ncrkhs, BMV-int}.
\end{rem}

Similarly, we have a generalized analog of the half-plane, which by the Effros-Winkler separation theorem, includes all matrix convex sets.

\begin{defn} \label{def:H_L}
Let $\cH$ be a Hilbert space. Let $\Omega \subset \M(E)$ be an nc domain invariant under pointwise conjugation and $L \colon \Omega \to \M(B(\cH))$ be an injective nc analytic function. We define $\HH_L= \left\{\bZ \in \Omega \mid \re(L(\bZ)) < I \right\}$. Following \cite{BMV-int}, we also have a representation $\Omega = \D_Q$, where $Q$ is obtained from $L$ using the Cayley transform. 
\end{defn}

Note that as observed by Taylor in \cite[Section 6]{Taylor-frame}, for every nc domain $\Omega \subset \M_d$, the algebra $\cT(\Omega)$ of holomorphic nc functions on $\Omega$ is a nuclear Frechet algebra. The topology on $\cT(\Omega)$ is the topology of uniform convergence on levelwise compacta. Note that the map that takes an nc function to its restriction to the various levels gives a continuous embedding $\cT(\Omega) \hookrightarrow \prod_{n=1}^{\infty} \cO(\Omega(n)) \otimes M_n$, where we write $\cO(\Omega(n))$ for the algebra of holomorphic functions on $\Omega(n)$ with the topology of uniform convergence on compacta. Since $\cO(\Omega(n)) \otimes M_n$ is nuclear for every $n \in \N$, the product is also nuclear by \cite[Proposition 50.1]{Treves-tvs}. Moreover, since the conditions of respecting direct sums and similarities are closed, we immediately see that $\cT(\Omega)$ is a closed subspace and is also nuclear. Hence, by \cite[Proposition 50.2]{Treves-tvs}, every bounded subset of a nuclear space is precompact, and thus, in particular, in a Frechet space, every bounded sequence has a convergent subsequence. In other words, we can apply Montel's theorem to sequences of nc functions and maps. The obtained convergence is, of course, uniform convergence on compacta on each level. There are stronger versions of Montel's theorem (see, for example, \cite[Theorem 13.14]{AgMcCYng-book} and \cite[Theorem 5.2]{Popescu-funcI}). However, this is the one that will be used almost exclusively throughout this paper. We will apply Popescu's version of Montel's theorem for the row ball \cite[Theorem 5.2]{Popescu-funcI} to obtain a stronger Denjpy-Wolff theorem in this setting.

\section{Noncommutative hyperbolic metrics} \label{sec:hup_metric}

Noncommutative hyperbolic metrics and the nc Lempert function were introduced by the first author and Vinnikov in \cite{BelVin17}. For a contraction $T \in B(\cH,\cK)$, we will write $D_T = 1 - T T^*$ and $D_{T^*} = 1 - T^* T$ for the defect operators. The nc Lempert function for an nc domain $\Omega \subset \M(E)$ is defined by
\[
\delta(\bX,\bY)(\bZ) = \left[ \sup\left\{ t \in [0,\infty] \mid \forall s\in [0,t), \begin{pmatrix} \bX & s \bZ \\ 0 & \bY \end{pmatrix} \in \Omega \right\} \right]^{-1}
\]
Here $\bX \in \Omega(n)$, $\bY \in \Omega(m)$, and $\bZ \in M_{n,m} \otimes E$. We also set for $\bX, \bY \in \Omega(n)$.
\[
\tdel(\bX,\bY) = \delta(\bX,\bY)(\bX - \bY).
\]
In the case of $\D_Q$ we have a special form for the nc Lempert function, which is a particular case of the equality appearing in \cite[Equation 21]{BelVin17}.

\begin{lem} \label{lem:D_Q_nc_lempert}
For $\bX \in \D_Q(n)$, $\bY \in \D_Q(m)$, and $\bZ \in M_{n,m} \otimes E$ we have
\[
\delta(\bX,\bY)(\bZ) = \left\| D_{Q(\bX)}^{-1/2} \Delta Q(\bX,\bY)(\bZ) D_{Q(\bY)^*}^{-1/2}\right\|.
\]
\[
\tdel(\bX,\bY) = \left\| D_{Q(\bX)}^{-1/2} \left(Q(\bX) - Q(\bY)\right) D_{Q(\bY)^*}^{-1/2}\right\|
\]
\end{lem}
\begin{proof}
Let $t > 0$ be such that $\left(\begin{smallmatrix} \bX & t \bZ \\ 0 & \bY \end{smallmatrix} \right) \in \D_Q$. Then
\[
\begin{pmatrix} Q(\bX) & t \Delta Q(\bX, \bY)(\bZ) \\ 0 & Q(\bY) \end{pmatrix} \begin{pmatrix} Q(\bX)^* & 0 \\ t \Delta Q(\bX,\bY)(\bZ)^* & Q(\bY)^* \end{pmatrix} < I.
\]
Multiplying the matrices and applying Schur's complement, we see that
\begin{multline*}
D_{Q(\bX)} - t^2 \Delta Q(\bX,\bY)(\bZ) \Delta Q(\bX,\bY)(\bZ)^* > \\ t^2 \Delta Q(\bX,\bY)(\bZ) Q(\bY)^* D_{Q(\bY)}^{-1} Q(\bY) \Delta Q(\bX, \bY)(\bZ)^*.
\end{multline*}
Let $T \in B(\cH,\cK)$ be a strict row contraction. Writing $D_T^{-1}$ as a power series in $T T^*$, we get that $T^* D_T^{-1} T = D_{T^*}^{-1} - I$. Hence,
\[
D_{Q(\bX)} > t^2 \Delta Q(\bX,\bY)(\bZ) D_{Q(\bY)^*}^{-1} \Delta Q(\bX, \bY)(\bZ)^*.
\]
Thus
\[
\frac{1}{t^2} > D_{Q(\bX)}^{-1/2} \Delta Q(\bX,\bY)(\bZ) D_{Q(\bY)^*}^{-1} \Delta Q(\bX, \bY)(\bZ)^* D_{Q(\bX)}^{-1/2}.
\]
Running the argument back, we see that this condition is necessary and sufficient. Hence we conclude that
\[
\delta(\bX,\bY)(\bZ) = \left\| D_{Q(\bX)}^{-1/2} \Delta Q(\bX,\bY)(\bZ) D_{Q(\bY)^*}^{-1/2}\right\|
\]
The second equality follows from the nc difference-differential formula \cite[Theorem 2.10]{KVV}.
\end{proof}

We have a similar result for generalized half-planes.

\begin{lem} \label{lem:H_L_nc_lempert}
For $\bX \in \HH_L(n)$, $\bY \in \HH_L(m)$, and $\bZ \in M_{n,m} \otimes \C^d$ we have
\[
\delta(\bX,\bY)(\bZ) = \frac{1}{2} \left\| \left(I - \re(L(\bX))\right)^{-1/2} \Delta L(\bX,\bY)(\bZ) \left( I - \re(L(\bY))\right)^{-1/2}\right\|.
\]
\[
\tdel(\bX,\bY) = \frac{1}{2} \left\| \left(I -\re(L(\bX))\right)^{-1/2} \left(L(\bX) - L(\bY)\right) \left( I - \re(L(\bY)) \right)^{-1/2}\right\|
\]
\end{lem}

\begin{rem}
As stated in \cite{BelVin17}, $\tilde{\delta}$ is symmetric. We include a brief proof of this fact for the sake of completeness. Let $\Omega$ be a unitarily invariant nc domain and let $t \in (0, \pi/2)$. For every $\bX,\, \bY \in \Omega(n)$, we have 
\[
\begin{pmatrix} \cos t & -\sin t \\ - \sin t & -\cos t \end{pmatrix} \begin{pmatrix} \bY & \frac{\cos t}{\sin t} (\bY - \bX) \\ 0 & \bX \end{pmatrix} \begin{pmatrix} \cos t & - \sin t \\ -\sin t & -\cos t \end{pmatrix} = \begin{pmatrix} \bX & \frac{\cos t}{\sin t} (\bX - \bY) \\ 0 & \bY \end{pmatrix}.
\]
Hence, $\tdel(\bX,\bY) = \tdel(\bY,\bX)$.
\end{rem}

We record the following simple observation, that is, however, using the power of the nc assumption (see also \cite[Theorem 5]{AbdKV-fixed}).

\begin{lem} \label{lem:fixed_point_scalar}
Let $\Omega \subset \M(E)$ be a bounded matrix convex nc domain that contains $0$. Let $f \colon \Omega \to \Omega$ be an nc analytic map. If there exists $0< r < 1$, such that $f(\Omega) \subset r\Omega$, then $f$ has a unique fixed point in $\Omega(1)$. Namely, there exists $\bX \in \Omega(1)$, such that $f(\bX) =\bX$ and the only fixed points of $f$ are ampliations of $\bX$.
\end{lem}
\begin{proof}
Let $f_n = f|_{\Omega(n)}$. By the Earle-Hamilton theorem, $f_n \colon \Omega(n) \to \Omega(n)$ has a unique fixed point. In particular, let $\bX \in \Omega(1)$ be the unique fixed point on this level, then $\bX^{\oplus n}$ is fixed by $f_n$ since $f$ respects direct sums. Thus, by uniqueness, the fixed point on level $n$ is $\bX^{\oplus n}$.
\end{proof}

\begin{lem} \label{lem:finsler_seminorm}
Let $\Omega = \D_Q$. Then for every $n$, $\rho_{\bX}(\bZ) = \delta(\bX,\bX)(\bZ)$ is a seminorm. Furthermore, $\rho_{\bX}$ is a norm if and only if $\Delta Q(\bX,\bX)$ is injective.
\end{lem}
\begin{proof}
Obvious from Lemma \ref{lem:D_Q_nc_lempert} and the linearity of $\Delta Q(\bX,\bX)$.
\end{proof}

Note that for every $\bX \in \D_Q(n)$, we can define a collection of seminorms on $M_k(M_n(E))$, via $\rho_{k,\bX}(\bZ) = \delta(I_k \otimes \bX,  I_k \otimes \bX)(\bZ)$.

\begin{prop}
Let $Q$ be such that for some $\bX \in \D_Q(n)$, $\Delta Q(\bX, \bX)$ is injective. Then $\rho_{k,\bX}$ form an operator space structure on $M_n(E)$.
\end{prop}
\begin{proof}
By our assumption and the properties of the nc derivative $\Delta Q(I_k \otimes \bX, I_k \otimes \bX)$ is injective for every $k \in \N$. Thus, each $\rho_k$ is a norm by the previous lemma. Hence, we will verify Ruan's axioms.

Let $\bZ \in\M_{k n}(E)$ and $\bW \in M_{\ell n}(E)$. By the properties of $\delta$, we have that
\[
\rho_{k+\ell}(\bZ \oplus \bW) = \delta(I_{k+\ell} \otimes \bX, I_{k+\ell} \otimes \bX)(\bZ \oplus \bW) = \max\{\rho_k(\bZ),\rho_{\ell}(\bW)\}.
\]

Now let $\alpha \in M_{\ell,k}$, $\beta \in M_{k,\ell}$, and $Z \in M_k (M_{n}(E))$. Write $I = I_{M_n(E)}$ and observe that since $\Delta Q(\bX, \bX)$ is linear and satisfies for every $m \in \N$, $\Delta Q(I_m \otimes \bX, I_m \otimes \bX) = I_{M_m} \otimes \Delta Q(\bX, \bX)$, we get that
\[
\Delta Q(I_{\ell} \otimes \bX, I_{\ell} \otimes \bX)((\alpha \otimes I) \bZ (\beta \otimes I)) = (\alpha \otimes I) \Delta Q(I_k \otimes \bX, I_k \otimes \bX)(\bZ) (\beta \otimes I).
\]
Hence,
\[
\rho_{\ell,\bX}\left((\alpha \otimes I)\bZ (\beta \otimes I)\right) \leq \|\alpha\|\|\beta\|\rho_{k,\bX}(\bZ).
\]
This completes the proof.
\end{proof}

\begin{rem}
If $Q \colon \M(E) \to \M(B(\cH,\cK))$ is an injective analytic nc function, it is not hard to check that $\Delta Q(\bX, \bX)$ is injective for every $\bX \in \M(E)$. Indeed, for every $\bZ \in M_n(E)$, such that $\Delta Q(\bX, \bX)(\bZ) = 0$,
\[
Q \left(\begin{pmatrix} \bX & \bZ \\ 0 & \bX \end{pmatrix} \right) = \begin{pmatrix}
    Q(\bX) & 0 \\ 0 & Q(\bX)
\end{pmatrix} = Q(\bX \oplus \bX).
\]
Hence, for every point $\bX \in \D_Q$, we have an operator space structure on the ``tangent space'' at $\bX$.
\end{rem}

\begin{cor}
Let $f \colon \D_Q \to \D_P$ be an nc function. Then, for every $\bX \in \D_Q$, the linear map $\Delta f(\bX,\bX)$ is a complete contraction.
\end{cor}
\begin{proof}
The claim follows immediately from \cite[Proposition 3.2]{BelVin17} and the properties of $\delta$ and the nc difference-differential operator.
\end{proof}

Recall from \cite{BelVin17} that there are several ways to define distances on an nc domain $\Omega$ using $\widetilde{\delta}$. The simplest one is a level-dependent distance. Let $\bX, \bY \in \Omega(n)$, we set:
\[
\widetilde{d}_{\Omega}(\bX,\bY) = \inf\{ \sum_{k=1}^m \widetilde{\delta}(\bZ_{k-1},\bZ_k) \mid m \in \N,\, \bX = \bZ_0,\bZ_1,\ldots,\bZ_m = \bY \in \Omega(n) \}
\]
One can remove the dependence on levels by altering the definition slightly
\[
\widetilde{d}_{\Omega,\infty}(\bX, \bY) = \inf\{ \sum_{k=1}^m \widetilde{\delta}(\bZ_{k-1},\bZ_k) \mid m, \ell \in \N,\,I_{\ell} \otimes \bX = \bZ_0,\bZ_1,\ldots,\bZ_m = I_{\ell} \otimes \bY \in \Omega(\ell n) \}
\]
Lastly, one can define a distance by measuring the lengths of paths connecting the two points.
\[
d_{\Omega}(\bX,\bY) = \inf\{ \int_0^1 \delta(\gamma(t),\gamma(t))(\gamma'(t))dt \mid \gamma \in C^1_{pw}([0,1],\Omega(n))\}.
\]
\[
d_{\Omega,\infty}(\bX,\bY) = \inf\{ \int_0^1 \delta(\gamma(t),\gamma(t))(\gamma'(t))dt \mid \gamma \in C^1_{pw}([0,1],\Omega(\ell n)), \ell \in \N\}.
\]

Here $C^1_{pw}([0,1],\Omega(n))$ stands for the piecewise continuously differentiable paths in $\Omega(n)$. Namely $\gamma \colon [0,1] \to \Omega(n)$, such that there exists finitely many $0 = t_1 < t_2 < \cdots < t_k = 1$, such that $f$ is $C^1$ on every open segment $(t_i,t_{i+1})$ and the one-sided derivatives exist at all the $t_i$.

\begin{lem}  \label{lem:delta_and_isom}
Let $\Omega \subset \M(E)$ be an open bounded matrix convex nc domain. Let $V \colon \C^k \to \C^n$ and $W \colon \C^{\ell} \to \C^m$ be isometries. Then, for every $\bX \in \Omega(n)$, $\bY \in \Omega(m)$,  and $\bZ \in M_{n,m}(E)$, we have
\[
\delta(V^* \bX V, W^*\bY W)(V^*\bZ W) \leq \delta(\bX,\bY)(\bZ ).
\]
In particular, if $\bX,\bY \in \Omega(n)$, then for every isometry $V \colon \C^k \to \C^n$:
\[
\widetilde{\delta}(V^* \bX V,V^* \bY V) \leq \widetilde{\delta}(\bX,\bY).
\]
\end{lem}
\begin{proof}
Note that for $t > 0$
\[
\begin{pmatrix} V^* & 0 \\ 0 & W^* \end{pmatrix} \begin{pmatrix} \bX  &  t\bZ \\ 0 & \bY \end{pmatrix} \begin{pmatrix}  V & 0 \\ 0 & W \end{pmatrix} = \begin{pmatrix} V^* \bX V & t V^* \bZ W \\ 0 & W^* \bY W \end{pmatrix}.
\]
Hence, if $\left( \begin{smallmatrix} \bX  &  t \bZ  \\ 0 & \bY  \end{smallmatrix} \right) \in \Omega(n+m)$, then since $\Omega$ is matrix convex and $V \oplus W$ is an isometry, we conclude that $\left( \begin{smallmatrix} V^* \bX V & t V^* \bZ W \\ 0 & W^* \bY W \end{smallmatrix} \right) \in \Omega(n+m)$. Hence
\begin{multline*}
A = \left\{ t \in [0,\infty] \mid \begin{pmatrix} \bX  &  s \bZ \\ 0 & \bY \end{pmatrix} \in \Omega(n+m),\, \text{for all } s\in [0,t) \right\} \subseteq \\ \left\{ t \in [0,\infty] \mid \begin{pmatrix} V^* \bX V & s V^* \bZ W \\ 0 & W^* \bY W \end{pmatrix} \in \Omega(n+m),\, \text{for all } s\in [0,t) \right\} =B.
\end{multline*}
Since $\delta(V^* \bX V, W^* \bY W)(V^* \bZ W) = \left(\sup B \right)^{-1}$ and $\delta(\bX,\bY)(\bZ) = \left(\sup A \right)^{-1}$, we get the desired inequality.

The second inequality follows immediately from the first.
\end{proof}

\begin{cor} \label{cor:no_infty_needed}
Let $\Omega$ be an open matrix convex domain and let $\bX, \bY \in \Omega(n)$, then
\[
\widetilde{d}_{\Omega}(\bX,\bY) = \widetilde{d}_{\Omega,\infty}(\bX,\bY),\, d_{\Omega}(\bX,\bY) = d_{\Omega,\infty}(\bX,\bY).
\]
\end{cor}
\begin{proof}
By definition $\widetilde{d}_{\Omega,\infty}(\bX,\bY) \leq \widetilde{d}_{\Omega}(\bX,\bY)$ and $d_{\Omega,\infty}(\bX,\bY) \leq d_{\Omega}(\bX,\bY)$ (see also \cite{BelVin17}). Now let $\epsilon >0$ be arbitrary. We can find $\ell, m \in \N$ and $I_{\ell} \otimes \bX = \bZ_0,\bZ_1,\ldots,\bZ_m = I_{\ell} \otimes \bY \in \Omega(\ell n)$, such that
\[
\sum_{k=1}^m \widetilde{\delta}(\bZ_{k-1},\bZ_k) < \widetilde{d}_{\Omega,\infty}(\bX,\bY) + \epsilon.
\]
Let $V \colon \C^n \to \C^{\ell n}$ be the isometry that embeds $\C^n$ as the subspace spanned by the first $n$ standard basis vectors. Then $V^* \bZ_0 V = \bX$, $V^* \bZ_m V = \bY$, and since $\Omega$ is matrix convex for every $1 \leq k \leq m-1$, $V^* \bZ_k V \in \Omega(n)$. By Lemma \ref{lem:delta_and_isom} we have that for every $1 \leq k \leq m$, $\widetilde{\delta}(V^* \bZ_{k-1} V, V^* \bZ_k V) \leq \widetilde{\delta}(\bZ_{k-1},\bZ_k)$ and thus
\[
\sum_{k=1}^m \widetilde{\delta}(V^*\bZ_{k-1} V,V^*\bZ_k V) < \widetilde{d}_{\Omega,\infty}(\bX,\bY) + \epsilon.
\]
This proves that $\widetilde{d}_{\Omega,\infty}(\bX,\bY) = \widetilde{d}_{\Omega}(\bX,\bY)$.

Similarly, let $\gamma \in C^1_{pw}([0,1],\Omega(\ell n))$ be such that $\gamma(0) = I_{\ell} \otimes \bX$, $\gamma(1) = I_{\ell} \otimes \bY$, and
\[
\int_0^1 \delta(\gamma(t),\gamma(t))(\gamma'(t)) < d_{\Omega,\infty}(\bX,\bY) + \epsilon.
\]
Let $V$ be as above and set $\beta = V^* \gamma V \in C^1_{pw}([0,1],\Omega(n))$. Then, $\beta(0) = \bX$, $\beta(1) = \bY$, and $\beta'(t) = V^* \gamma'(t) V$. Thus by Lemma \ref{lem:delta_and_isom}, 
\[
\int_0^1 \delta(\beta(t),\beta(t))(\beta'(t)) \leq \int_0^1 \delta(\gamma(t),\gamma(t))(\gamma'(t)) < d_{\Omega,\infty}(\bX,\bY) + \epsilon.
\]
Thus $d_{\Omega,\infty}(\bX,\bY) = d_{\Omega}(\bX,\bY)$.
\end{proof}

\begin{lem} \label{lem:delta_conv_comb}
Let $\Omega$ be an open matrix convex domain. Let $\bX, \bY \in \Omega(n)$ and $\bA \in M_n(E)$. Then for every $0 < t < 1$
\[
\delta(t \bX +(1-t)\bY, t\bX + (1-t) \bY)(\bA) \leq \frac{1}{2 \sqrt{t (1-t)}}\max\{ \delta(\bX,\bY)(\bA),\delta(\bY, \bX)(\bA)\}.
\]
\[
\delta(t \bX +(1-t)\bY, t\bX + (1-t) \bY)(\bA) \leq \max\{ \delta(\bX,\bX)(\bA), \delta(\bY,\bY)(\bA) \}.
\]
\end{lem}
\begin{proof}
Let us define for every $0 < t < 1$, an isometry $V_t \colon \C^n \to \C^{2n}$:
\[
V_t = \begin{pmatrix} \sqrt{t} I_n \\ \sqrt{1-t} I_n \end{pmatrix}
\]
Then
\[
t\bX + (1-t)\bY = V_t^* (\bX \oplus \bY) V_t
\]
 Applying Lemma \ref{lem:delta_and_isom} and \cite[Lemma 3.5]{BelVin17}, we get
\begin{multline*}
\delta(V_t^* (\bX \oplus \bY) V_t, V_t^* (\bX \oplus \bY) V_t)\left( \frac{1}{2 \sqrt{t (1-t)}}V_t^*  \begin{pmatrix} 0 & \bA \\ \bA & 0 \end{pmatrix} V_t \right) \leq \\  \frac{1}{2 \sqrt{t (1-t)}} \delta(\bX \oplus \bY, \bX \oplus \bY)\left( \begin{pmatrix} 0 & \bA \\ \bA & 0 \end{pmatrix} \right) =  \frac{1}{2 \sqrt{t (1-t)}}\max\{ \delta(\bX,\bY)(\bA),\delta(\bY, \bX)(\bA)\}.
\end{multline*}
The proof of the second inequality is identical. We only note that $\bA = V_t^* (\bA \oplus \bA) V_t$.
\end{proof}

\begin{thm} \label{thm:conv_one_dist}
Let $\Omega$ be an open matrix convex domain. Then for every $\bX, \bY \in U(n)$, we have that
\[
\widetilde{d}_{\Omega}(\bX,\bY) \leq d_{\Omega}(\bX,\bY) \leq \pi \widetilde{d}_{\Omega}(\bX,\bY).
\]
In other words, the metrics $d_{\Omega}$ and $\widetilde{d}_{\Omega}$ are equivalent.
\end{thm}
\begin{proof}
By \cite{BelVin17}, $\widetilde{d}_{\Omega}(\bX,\bY) \leq d_{\Omega}(\bX,\bY)$. Now let $\epsilon >0$ and find $\bX = \bZ_0,\bZ_1,\ldots,\bZ_m = \bY \in \Omega(n)$, such that
\[
\sum_{k=1}^m \widetilde{\delta}(\bZ_{k-1},\bZ_k) < \widetilde{d}_{\Omega}(\bX,\bY) + \frac{\epsilon}{\pi}.
\]
Now consider the paths $\beta_k(t) = t \bZ_k + (1-t) \bZ_{k-1}$ and concatenate them to obtain a piecewise smooth path $\beta \colon [0,1] \to \Omega(n)$, such that $\beta(0) =\bX$ and $\beta(1) = \bY$. By Lemma \ref{lem:delta_conv_comb} we have that
\[
\int_0^1 \delta(\beta_k(t),\beta_k(t))(\beta_k'(t)) dt \leq \widetilde{\delta}(\bZ_{k-1},\bZ_k) \int_0^1 \frac{dt}{2 \sqrt{t(1-t)}} = \pi \widetilde{\delta}(\bZ_{k-1},\bZ_k)
\]
Thus,
\[
\int_0^1 \delta(\beta(t),\beta(t))(\beta'(t))dt < \pi \widetilde{d}_{\Omega}(\bX,\bY) + \epsilon.
\]
Since $\epsilon$ was arbitrary, we get the result.
\end{proof}

Now we collect a few useful properties of the distance $d_{\Omega}$.

\begin{prop} \label{lem:metric_properties}
Let $\Omega$ be a uniformly bounded and open matrix convex domain and $\bX, \bY, \bZ,\bW \in \Omega(n)$. Then
\begin{itemize}
\item[(i)] For $V \in U_n$ a unitary, we have $d_{\Omega}(V^* \bX V, V^* \bY V) = d_{\Omega}(\bX,\bY)$.

\item[(ii)] If $V \colon \C^k \to \C^n$ is an isometry, then $d_{\Omega}(V^* \bX V, V^* \bY V) \leq d_{\Omega}(\bX, \bY)$

\item[(iii)] $d_{\Omega}(\bX \oplus \bY, \bZ \oplus \bW) \geq \max\{ d_{\Omega}(\bX,\bZ), d_{\Omega}(\bY,\bW)\}$.
\end{itemize}
\end{prop}
\begin{proof}
Given $\epsilon > 0$, let $\gamma \colon [0,1] \to \Omega(n)$ be a piecewise continuously differentiable path, such that $\gamma(0) = \bX$, $\gamma(1) = \bY$, and 
\[
\int_0^1 \delta(\gamma(t),\gamma(t))(\gamma'(t))dt < d_{\Omega}(\bX,\bY) + \epsilon.
\]
Now note that $\delta(V^* \gamma(t) V, V^* \gamma(t) V)(V^* \gamma'(t) V) = \delta(\gamma(t),\gamma(t))(\gamma'(t))$ by \cite{BelVin17}. Thus taking the path $V^* \gamma V$, we get that $d_{\Omega}(V^* \bX V, V^* \bY V) < d_{\Omega}(\bX,\bY) + \epsilon$. Since $\epsilon > 0$ was arbitrary, we get that $d_{\Omega}(V^* \bX V, V^* \bY V) \leq d_{\Omega}(\bX,\bY)$. Since $\operatorname{Ad}_{V^*}$ is invertible, we get the converse inequality. 

The second item follows immediately from Lemma \ref{lem:delta_and_isom}. For the third item, let $V \colon \C^n \to \C^{2n}$ be an isometry that embeds $\C^n$ as the span of the first $n$ standard basis vectors. Then 
\[
d_{\Omega}(\bX,\bZ) = d_{\Omega}( V^* (\bX \oplus \bY) V, V^* (\bZ \oplus \bW) V) \leq d_{\Omega}(\bX \oplus \bY,\bZ \oplus \bW).
\]
Similarly $d_{\Omega}(\bY,\bW) \leq d_{\Omega}(\bX \oplus \bY,\bZ \oplus \bW)$.
\end{proof}

\section{Noncommutative Denjoy-Wolff theorems} \label{sec:denjoy_wolff}

Let $f \colon \D \to \D$ be a holomorphic function that is not an automorphism and has no fixed points. In order to prove the classical Denjoy-Wolff theorem, one applies the Montel theorem to obtain convergent subsequences of the sequence $f_n$ of iterations of $f$. We can then split the rest of the proof into two pieces. In the first part, we apply the maximum modulus principle to show that a partial limit of the sequence $f_n$ is either constant or maps $\D$ into itself. Then one applies the Wolff lemma to show that in the constant case, the value is the Wolff point of $f$ and works to rule out the second option. This observation leads us to try and extend the Denjoy-Wolff-type result of Abate and Raissy to the noncommutative setting.

\begin{defn}
Let $\Omega \subset \M_d$ be an nc domain. We will say that $\Omega$ has a simple nc boundary if for every nc holomorphic function $f \colon \Omega' \subset \M_k \to \overline{\Omega}$, if $f(\Omega') \cap \partial \Omega \neq \emptyset$, then $f$ is constant.
\end{defn}

\begin{example}
By the maximum modulus principle \cite[Lemma 6.1]{3S1}, the row ball $\fB_d$ has a simple nc boundary. 
\end{example}

The following is an extension of the result of Abate and Raissy \cite[Theorem 2]{AbateRaissy} to the noncommutative setting.

\begin{thm}\label{thm:nc_denjoy_wolff}
Let $\Omega \subset \M_d$ be a uniformly bounded matrix convex domain with a simple nc boundary. Assume that $\Omega(1)$ is strictly convex. Let $f \colon \Omega \to \Omega$ be an nc map with no fixed points. Then there exists $\zeta \in \partial \Omega(1)$, such that the sequence of iterates $f^{\circ\, n}$ converges to the constant function $\zeta$, where $$f^{\circ\, n} = \underbrace{f \circ f \circ \cdots \circ f}_{n \text{ times}}.$$
\end{thm}
\begin{proof}
By \cite[Theorem 2]{AbateRaissy}, the sequence of iterates $(f^{\circ\, n})_1 = f_1^{\circ\, n}$ converges to the constant function $\zeta$. Now applying Montel's theorem to $f^{\circ\, n}$, we know that the sequence has convergent subsequences. Let $f^{\circ\, n_k}$ be a convergent subsequence and $g$ its limit. Since $g_1 = \zeta$, by the fact that $\Omega$ has a simple nc boundary, $g$ is constant and equals $\zeta$. Hence, every convergent subsequence converges to the same function, and thus, $f^{\circ\, n}$ converges to $\zeta$ uniformly on every compact.
\end{proof}

\begin{thm} \label{thm:nc_denjoy_wolff_ball}
Let $f \colon \fB_d \to \fB_d$ be an nc function without fixed points, then the sequence of iterates $f^{\circ\, n}$ converges uniformly on subball $r\overline{\fB_d}$ to the constant function $\zeta$, where $\zeta $ is the Denjoy-Wolff point of $f_1 = f|_{\fB_d(1)}$.
\end{thm}
\begin{proof}
By Popescu's Montel theorem \cite[Theorem 5.2]{Popescu-funcI}, the sequence $f^{\circ\, n}$ has subsequences that converge uniformly on subballs. Hence, we can argue as above, but the convergence is on subballs.
\end{proof}

\subsection{Maximal quantization of the unit ball}

Let $K \subset \C^d$ be a bounded open convex set and assume for simplicity that $0 \in K$. By the Hahn-Banach separation theorem, for every $z \in \C^d \setminus K$, there exists $w \in \C^d$ and $s \in \R$, such that for every $u \in K$, $\re \langle u, w \rangle < s \leq \re \langle z, w \rangle$. Since necessarily $s > 0$, we may replace $w$ by $\frac{1}{s} w$ and set $s = 1$. Hence, $K$ can be identified as an intersection of open half-planes. Let us write $K^{\circ} = \left\{ w \in \C^d \mid \forall u \in K,\, \re \langle u, w \rangle \leq 1 \right\}$. As in \cite{DDSS17} (in the compact case), we define the maximal matrix convex set that contains $K$ by $\Max(K) = \sqcup_{n \in \N} \Max(K)(n)$, where
\[
\Max(K)(n) = \left\{ \bZ \in \M_d(n) \mid \forall w \in K^{\circ},\, \re\left( \sum_{j=1}^d \overline{w_j} Z_j \right) < 1 \right\}.
\]
From now will focus on $\Max(\BB_d)$. Let $\overline{\Max(\BB_d)}$ be the levelwise closure of $\Max(\BB_d)$. Note that by continuity of the functionals, $\overline{\Max(\BB_d)} \subset \Max(\overline{\BB_d})$. Conversely, if $Z \in \Max(\overline{\BB_d})$, then for every $0 < r < 1$, $r Z \in \Max(\BB_d)$ and, thus, these sets coincide. We aim to show that the Denjoy-Wolff theorem holds for $\Max(\BB_d)$. One of the key ingredients is the maximal modulus principle. We need some preparation to prove a version of the maximum modulus principle for $\Max(\BB_d)$.

\begin{lem} \label{lem:real_part_const}
Let $\Omega \subset \M_d$ be an open nc set and $f \colon \Omega \to \M_1$ be a locally bounded nc function. If $\re f(\bW) = I$ for all $\bW \in \Omega$, then $f$ is constant.
\end{lem}
\begin{proof}
Since $f$ is levelwise holomorphic and $\Delta f(\bW, \bW)(\bZ)$ is the directional derivative at $\bW$. It suffices to prove that for every $\bW \in \Omega$ and $\bZ \in \M_d$, $\Delta f(\bW, \bW)(\bZ) = 0$. Note that since $\Omega$ is open there exists $t > 0$, such that $\left( \begin{smallmatrix} \bW & t\bZ \\ 0 & \bW \end{smallmatrix}\right) \in \Omega$. Hence,
\[
\begin{pmatrix} I & 0 \\ 0 & I \end{pmatrix} = \re f\left( \begin{pmatrix} \bW & \bZ \\ 0 & \bW \end{pmatrix} \right) = \begin{pmatrix} I & \frac{1}{2} \Delta f(\bW, \bW)(\bZ) \\ \frac{1}{2} \Delta f(\bW, \bW)(\bZ)^* & I \end{pmatrix}.
\]
\end{proof}

\begin{lem} \label{lem:max_boundary}
Let $\bX \in \Max(\overline{\BB_d})$ and assume that $X_1 = I$. Then $X_2 = \cdots = X_d = 0$.
\end{lem}
\begin{proof}
Since $\Max(\overline{\BB_d})$ is matrix convex, for every unit vector $u \in \C^n$, we have that $u^* \bX u = (1 , \langle X_2 u, u \rangle, \cdots, \langle X_d u, u \rangle) \in \overline{\BB_d}$. Which implies that the numerical ranges of $X_2,\, \ldots,\, X_d$ are $\{0\}$ and thus these matrices are zero.

\end{proof}

The following theorem is a version of the maximum modulus from \cite{3S1} for the set $\Max(\BB_d)$.
\begin{thm} \label{thm:max_modulus_max}
Let $\Omega \subset \M_k$ be an open unitarily-invariant nc set and let $f \colon \Omega \to \Max(\overline{\BB_d})$ be a locally bounded holomorphic function. If for some $\bW_0 \in \Omega$, $f(\bW_0) \in \partial \Max(\overline{\BB_d})$, then $f$ is constant.
\end{thm}
\begin{proof}
Assume $\bW_0 \in \Omega(n)$. Since $f(\bW_0) \in \partial \Max(\overline{\BB_d})$, there exists $\lambda \in \C^d$ with $\|\lambda\| = 1$, such that $\re \left(\sum_{j=1}^d \overline{\lambda_j} f_j(\bW_0)\right)$ has $1$ in its spectrum. Set $g(\bW) = \sum_{j=1}^d \overline{\lambda_j} f_j(\bW)$ and let $u$ be a unit vector, such that $\re g(\bW_0) u = u$. Let $k \in \N$ and define $\varphi_k \colon \Omega(nk) \to \R$ by $\varphi_k(\bW) = \re \langle g(\bW) u^{(k)}, u^{(k)} \rangle$, where $u^{(k)} = \frac{1}{\sqrt{k}} \left( \begin{smallmatrix} u \\ \vdots \\ u \end{smallmatrix} \right) \in \C^{nk}$. Since $\re g(\bW) \leq I$ for all $\bW \in \Omega$, we have that $\varphi_k(\bW) \leq 1$. Moreover, by assumption $\varphi_k(I_k \otimes \bW_0) = 1$. Since $\varphi_k$ is pluriharmonic, we can apply the classical maximum modulus to deduce that $\varphi_k$ is constant. Hence, for every $\bW \in \Omega(nk)$, we have that $\re \langle g(\bW) u, u \rangle = 1$. By unitary invariance of $\Omega$, we have that for every unitary $nk \times nk$ matrix $V$, we have that 
\[
\re \langle g(\bW) V u^{(k)}, g(\bW) Vu^{(k)} \rangle = \re \langle g(V^* \bW V) u^{(k)}, u^{(k)} \rangle  = 1.
\]
In particular, the numerical range of $\re g(\bW)$ is the singleton $1$, i.e, $\re g(\bW) = I$. Now we can apply Lemma \ref{lem:real_part_const} to deduce that the restriction of $g$ to levels that are multiples of $n$ is constant. However, for every $m$ not divisible by $n$, we have that $\Omega(m)$ can be embedded into $\Omega(nm)$ via direct sums, and thus $g$ is constant. 

Now note that if $U \in M_d$ is a unitary, then for $\bW \in \Max(\overline{\BB_d})$, we have $U(\bW) \in \Max(\overline{\BB_d})$. Hence, applying a unitary, we may assume that $\lambda$ is the first vector of the standard basis. Hence, we have that $f_1 = 1$. By Lemma \ref{lem:max_boundary}, we have that $f_2 = \cdots = f_d = 0$.
\end{proof}

\begin{cor} \label{cor:nc_denjoy_wolff_max}
Let $f \colon \Max(\BB_d) \to \Max(\BB_d)$ be an nc map without fixed points. Then, the sequence of iterates of $f$ converges to the constant function $\zeta$, where $\zeta$ is the Denjoy-Wolff point of $f_1$.
\end{cor}
\begin{proof}
The above theorem shows that $\Max(\BB_d)$ has a simple nc boundary. Thus, the result is a consequence of Theorem \ref{thm:nc_denjoy_wolff}
\end{proof}

\begin{question}
Does every quantization of $\BB_d$ have a simple nc boundary?
\end{question}

\section{Noncommutative Vigu\'{e} and Bedford theorems} \label{sec:nc_vigue}

In \cite{Vigue-tams}, Vigu\'{e} has proved that for a bounded convex domain $U \subset \C^d$, the set of fixed points of an analytic self-map is a smooth analytic subvariety. Moreover, a smooth analytic subvariety $V \subset U$ is the set of fixed points of some analytic self-map if and only if $V$ is a holomorphic retract of $U$, i.e., there exists $\psi \colon U \to V$ analytic, such that $\psi|_V = \mathrm{id}_V$. Our goal here is to show that Vigu\'{e}'s methods extend almost verbatim to the setting of bounded matrix convex sets.

By \cite[Proposition 2.1.8]{Abate-book} \cite[Corollary 2.1.11]{Abate-book}, every bounded convex domain $W \subset \C^n$ is both taut and tautly embedded in $\C^n$, Hence, given a sequence of maps $f_n \colon W \to W$ that converges uniformly on compacta to $f$, we have that $f(W) \subset W$ or $f(W) \subset \partial W$.

Note that if $\Omega \subset \M_d$ is matrix convex, then for decomposable point in $\Omega$, the restriction to the invariant subspaces and the compression to coinvariant subspaces are in $\Omega$. We say that an nc subset $\Xi \subset \Omega$ is (relatively) bifull if it is closed under similarities, restrictions to invariant subspaces, and compressions to coinvariant subspaces (see \cite{BMV-int}). To be more precise,``closed under similarities'' means that for every $n\in \N$, $X \in \Xi(n)$, and every $ S \in \mathrm{GL}_n(\C)$, such that $S^{-1} X S \in \Omega$, we have that $S^{-1} X S \in \Xi$.

\begin{thm}[Noncommutative Vigu\'{e}'s theorem] \label{thm:nc_vigue}
Let $\Omega \subset \M_d$ be a bounded matrix convex set. Let $\fV \subset \Omega$ be an nc relatively bifull subset, such that for every $n \in \N$, $\fV(n)$ is a smooth analytic subvariety of $\Omega(n)$. Then, $\fV$ is the fixed point set of an nc map $f \colon \Omega \to \Omega$, if and only if $\fV$ is the image on an nc holomorphic retraction, i.e., there exists $\psi \colon \Omega \to \Omega$, such that $\psi \circ \psi = \psi$ and $\psi(\Omega) = \fV$.
\end{thm} 
\begin{proof}
Just like classically, one direction is easy. If there exists an nc holomorphic retraction $\psi \colon \Omega \to \fV$, then $\fV$ is the fixed point set of $\psi$. Moreover, every fixed point set is level-wise smooth by the classical theorem of Vigu\'{e} and is bifull by basic properties of nc functions.

For the second direction, we can apply Vigu\'{e}'s proof of his theorem and only verify small details. Let $f$ be an nc map, so $\fV$ is the fixed point set of $f$. Write $f_0 = \mathrm{id}_U$ and $f_n = f \circ f_{n-1}$, for $n\in\N$. Consider the sequence $\varphi_m = \frac{1}{m} \sum_{k=0}^{m-1} f_k$. Since each $\Omega(n)$ is convex, each $\varphi_m$ is a self-map of $\Omega$. Passing to a subsequence, we may assume that $\varphi_m$ converges level-wise on compacta to an nc function $\varphi$. The fixed point set of $\varphi$ contains $\fV$. Again, since each $\Omega(n)$ is convex and bounded, it is taut, and thus $\varphi$ is also a self-map of $\Omega$. Now let $\psi$ be a partial limit of the sequence of iterates of $\varphi$. Refining the subsequence, if necessary, we see that $\psi$ is a retraction. Arguing as in the proof of \cite[Theorem 6.5]{Vigue-tams}, we see that level-wise $\psi_n$ maps $\Omega(n)$ into $\fV(n)$ and we are done.
\end{proof}

\begin{rem}
One can also argue level-wise instead of discussing nuclearity. Since on each level the Montel theorem holds, if we have a sequence of nc functions $f_n \colon \Omega \to \Omega$, then we can choose a subsequence $f_{n_k^{(1)}}$ that converges on the first level and a subsequence of it $f_{n_k^{(2)}}$ that converges on the second level and so on. Now we set $g_k = f_{n_k^{(k)}}$, and this sequence converges on all levels. The limit is, of course, an nc function.
\end{rem}

\begin{cor} \label{cor:no_level_1_fixed}
Let $\Omega$ be a bounded matrix convex set, and let $f \colon \Omega \to \Omega$ be a nc map. Then the fixed point set of $f$ contains scalar points. In other words, if $f$ has no fixed points on the first level, then $f$ has no fixed points.
\end{cor}
\begin{proof}
Let $\fV$ be the collection of fixed points of $f$. By the nc Vigu\'{e} theorem, we have that $\fV$ is the image of some nc holomorphic retraction. Since nc maps are graded, we have that $\fV(1) \neq \emptyset$.
\end{proof}

In \cite{Bedford}, Bedford proved an important dichotomy for converging iterates of self-maps of a taut domain. The following result is a straightforward generalization of Bedford's theorem to the setting of uniformly bounded matrix convex domains.

\begin{thm}[Noncommutative Bedford theorem] \label{thm:nc_bedford}
Let $\Omega \subset \M_d$ be a uniformly bounded matrix convex domain. Let $f \colon \Omega \to \Omega$ be an nc self-map. Let $f^{\circ n} = \underbrace{f\circ f\circ \cdots \circ f}_{n \text{ times}}$ be the sequence of iterates of $f$. Suppose the sequence $f^{\circ n_k}$ converges to $g$ level-wise on compacta. In that case, we have that either for every $n \in \N$, $g(\Omega) \subset \partial \Omega$ or there exists an nc analytic subvariety of $\fV \subset \Omega$, an nc automorphism $\varphi \colon \fV \to \fV$ and a holomorphic retraction $\psi \colon \Omega \to \fV$, such that $g = \varphi \circ \psi$.
\end{thm}
\begin{proof}
Due to tautness on every level, $g_n(\Omega(n))$ is either a subset of $\partial \Omega(n)$ or of $\Omega(n)$. Now assume that there exists $n_0 \in \N$, such that $g_{n_0}(\Omega(n_0)) \subset \partial \Omega(n_0)$. Let $\bX \in \Omega(1)$, then $I_{n_0} \otimes g(\bX) = g(I_{n_0} \otimes \bX) \in \partial \Omega(n_0)$. Since $\Omega$ is closed under direct sums and direct summands, $g(\bX) \in \partial \Omega(1)$. By tautness $g(\Omega(1)) \subset \partial \Omega(1)$. Thus, for every $n \in \N$, $I_n \otimes g(\bX) \in \partial \Omega(n)$ and applying tautness again we conclude that $g_n(\Omega(n)) \subset \partial \Omega(n)$.

Assume that $g(\Omega) \subset \Omega$. The proof proceeds as the proof of \cite[Theorem 1.1]{Bedford}. We record the proof here for the sake of completeness. Passing to subsequences, we may assume that $m_k = n_{k+1} - n_k$ and $\ell_k = n_{k+1} - 2 n_k$ are strictly increasing, and the sequences $f^{\circ m_k}$ and $f^{\circ \ell_k}$ converge (using the fact that the space of all level-wise holomorphic nc functions is a Montel space). Denote
\[
\lim_{k\to\infty} f^{\circ m_k} = \psi,\, \lim_{k\to \infty} f^{\circ \ell_k} = \tau.
\]
Since $f^{\circ m_k} \circ f^{\circ n_k} = f^{\circ n_{k+1}} = f^{\circ  n_k} \circ f^{\circ m_k}$, we have that $\psi\circ g = g = g \circ \psi$. Similarly, $\tau \circ g = \psi = g \circ \tau$. In particular, $\psi$ fixes the image of $g$. Applying the argument above we conclude that $\psi(\Omega) \subset \Omega$ and $\tau(\Omega) \subset \Omega$. Set $\fV = \fix(\psi)$ and note that $g(\Omega) \subset \fV$. For every $n \in \N$ and every $\bZ \in \Omega(n)$, we have $\ker  D g_n(\bZ) = \ker D\psi_n(\bZ)$ since by the chain rule $Dg_n(\psi_n (\bZ)) D\psi_n(\bZ) = Dg_n(\bZ)$ and $D\psi_n (\bZ) = D\tau_n (g_n(\bZ)) D g_n(\bZ)$. Note that $\fW(n) = \psi_n^{-1}(\fV(n))$ is an analytic subvariety of $\Omega(n)$ cut out by $\psi(\psi(\bZ)) - \psi(\bZ)$. The dimension of the fiber of $\fW(n)$ over every point $\bZ \in g_n(\Omega(n))$ is precisely $dn^2 - \dim \ker D\psi_n(\bZ) = dn^2 - \dim \ker Dg_n(\bZ)$. Thus $\fW(n)$ is of dimension $dn^2$, and we conclude that $\psi$ is a holomorphic retract with image $\fV$. Hence $\fV$ is a smooth nc analytic subvariety of $\Omega$.

Lastly note that for every $\bZ \in \fV$, $\tau(g(\bZ)) = g(\tau(\bZ)) = \psi(\bZ) = \bZ$. Hence $g(\fV) = \fV$ and $g|_{\fV}$ is an automorphism of $\fV$ with inverse $\tau|_{\fV}$. Set $\varphi = g|_{\fV}$ and thus $g = \varphi \circ \psi$.
\end{proof}

\subsection{Application to the isomorphism problem}

Recall that $H^{\infty}(\fB_d)$ is the operator algebra of bounded nc analytic functions on $\fB_d$. Moreover, $H^{\infty}(\fB_d)$ is the multiplier algebra of the noncommutative reproducing kernel Hilbert space of the nc Szego kernel. This implies that $H^{\infty}(\fB_d)$ is a dual operator algebra \cite{3S1}. The Hilbert space that $H^{\infty}(\fB_d)$ acts upon can be identified with the full Fock space $\cF^2_d$ \cite{BMV-ncrkhs,3S1}. Thus, $H^{\infty}(\fB_d)$ is unitarily equivalent to the weak-* closed subalgebra of operators on the Fock space $\cF^2_d$, generated by the left creation operators (the regular free semigroup algebra in the language of Davidson and Pitts \cite{DavPit-alg, DavPit-NPint, DavPit-inv}). 

Given a subset of bounded nc functions $S\subset H^{\infty}(\fB_d)$ we consider the nc analytic subvariety cut out by $S$:
\[
\fV(S) = \left\{ \bZ \in \fB_d \mid \forall f \in S,\, f(\bZ)= 0 \right\}.
\]
It is immediate that $\fV(S) = \fV(\cI)$, where $\cI$ is the weak-* closed two-sided ideal generated by $S$. Conversely, given a set of points $\fS \subset \fB_d$, we define
\[
\cI(\fS) = \left\{ f \in H^{\infty}(\fB_d) \mid f|_{\fS} = 0 \right\}.
\]
Clearly, $\cI(\fS)$ is weak-* closed two sided ideal in $H^{\infty}(\fB_d)$. In this section, we say that $\fV \subset \fB_d$ is an nc analytic subvariety if $\fV(\cI(\fV)) = \fV$ (in other words, $\fV$ is Zariski closed in the appropriate sense).

Given a non-empty analytic subvariety $\fV \subset \fB_d$, we have that $H^{\infty}(\fB_d)/\cI(\fV)$ is the algebra of bounded nc analytic functions on $\fV$. It can also be identified with the left multiplier algebra of an appropriate nc reproducing kernel Hilbert space on $\fV$. We will denote this algebra by $H^{\infty}(\fV)$.

The following results are necessary for our discussion:

\begin{thm}[{\cite[Theorem 6.12, Corollary 6.14]{3S1}}] \label{thm:3s1_isom}
Let $\fV \subset \fB_d$ and $\fW \subset \fB_e$ be two nc analytic varieties. Let $\alpha \colon H^{\infty}(\fV) \to H^{\infty}(\fW)$ be a completely isometric isomorphism. Then there exists nc maps $g \colon \fB_e \to \fB_d$ and $f \colon \fB_d \to \fB_e$, such that
\[
f \circ g|_{\fW} = \operatorname{id}_{\fW},\quad g \circ f|_{\fV} = \operatorname{id}_{\fV}.
\]
Moreover, composition with $g$ implements $\alpha$.
\end{thm} 

\begin{thm}[{\cite[Theorem 3.12]{Sha18}}] \label{thm:fixed_points}
Let $f \colon \fB_d \to \fB_d$ be an nc map, such that $f(0) = 0$. Let $\fix(f)$ be the set of fixed points of $f$. If $V \subset \C^d$ is a linear subspace, such that $V \cap \fB_d(1) = \fix(f)(1)$, then for every $n \in \N$ $\fix(f)(n) = \left(V \otimes M_n \right) \cap \fB_d(n)$.
\end{thm}
This theorem can be read differently. Recall that for an analytic self-map of $\fB_d(1)$ that fixes the origin, the fixed point set coincides with the fixed subspace of $f'(0)$ intersected with $\fB_d(1)$. Since for every $n \in \N$, $\Delta f(0_n,0_n) = f'(0) \otimes I_{M_n}$ is the derivative of $f_n$ at $0_n$. We see that the fixed point set of $f_n$ is precisely the set of fixed points of $\Delta f(0_n,0_n)$ intersected with $\fB_d(n)$.

Now with these theorems at hand, we can prove the following theorem, which is a strengthening of both \cite[Theorem 4.1]{Sha18} and \cite[Theorem 4.5]{DRS15}. Recall that for $S \subset \M_d$, the matrix span of $S$ is $\matsp(S) = \sqcup_{n \in \N} \matsp(S)(n)$, where
\[
\matsp(S)(n) = \operatorname{Span}\left\{ (I_d \otimes T)(\bX) \mid \bX \in S(n),\, T \in\cL(M_n)\right\}.
\]

\begin{thm} \label{thm:nc_DRS}
Let $\fV \subset \fB_d$ and $\fW \subset \fB_e$ be two nc analytic varieties. If $H^{\infty}(\fV)$ and $H^{\infty}(\fW)$ are completely isometrically isomorphic, then there exists $k \in \N$ and an automorphism $F\colon \fB_k \to \fB_k$, such that $\fV, \fW \subset \fB_k$ and $F$ maps $\fV$ onto $\fW$.
\end{thm}
\begin{proof}
Applying \cite[Lemma 8.2]{3S1} we can may assume that for $n >> 0$, we have that   $\matsp(\fV)(n) = \C^d \otimes M_n$ and $\matsp(\fW)(n) = \C^e \otimes M_n$. Let $g \colon \fB_e \to \fB_d$ and $f \colon \fB_e \to \fB_d$ be the nc maps guaranteed by Theorem \ref{thm:3s1_isom}. Set $h = g \circ f$ and observe that $\fV \subset \fix(h)$. However, by Corollary \ref{cor:no_level_1_fixed}, $\fix(h)(1) \neq \emptyset$. Pre- and post-composing with automorphisms of $\fB_d$ (they induce unitary equivalence on $H^{\infty}(\fB_d)$), we may assume that $0 \in \fix(h)(1)$. Thus, by Theorem \ref{thm:fixed_points}, there exists a subspace $V \subset \C^d$, such that $\fix(h)(n) = V \otimes M_n$. As mentioned above, $\fix(h)(n)$ coincides with the fixed points of the linear operator $\Delta f(0_n,0_n) = f'(0) \otimes I_{M_n}$. Therefore, $\matsp(\fV)(n) \cap \fB_d(n) \subset \fix(h)(n)$. Thus, by assumption on $\fV$, $\fix(h)(n) = \C^d \otimes M_n$, for every $n \in \N$. Applying the same argument to $f \circ g$, we see that $d =e$, and $f$ and $g$ are automorphisms of $\fB_d$ that map $\fV$ onto $\fW$ and vice versa.
\end{proof}

\section{Horospheres in generalized balls} \label{sec:horospheres}

Let us assume that $Q(0) = 0$ and $Q$ is nc analytic. For every $\bZ \in \D_Q$,
\[
\tdel(0,\bZ) = \left\| D_{Q(\bZ)}^{-1/2} Q(\bZ) \right\|.
\]
Let $T = Q(\bZ) Q(\bZ)^*$, then
\[
\tdel(0,\bZ)^2 =   \left\| T(I - T)^{-1} \right\| = \sup_{x \in \sigma(T)} \frac{x}{1-x}.
\]
Since the function $\frac{x}{1-x}$ is monotonically increasing on $[0,1)$, $\tdel(0,\bZ)^2 = \frac{\|T\|}{1 - \|T\|} = \frac{\|Q(\bZ)\|^2}{1 - \|Q(\bZ)\|^2}$. 

\begin{lem}
Let $T$ be a positive operator with $\|T\| < 1$, then $(1 - \|T\|)(1 - T)^{-1} \leq 1$.
\end{lem}
\begin{proof}
Since $T \leq \|T\|$, then $1 - T \geq 1 - \|T\|$. Now multiplying on left and right by $(1-T)^{-1/2}$, we get $1 \geq (1 - \|T\|)(1-T)^{-1}$.
\end{proof}

\begin{cor}
For every $\bZ \in \D_Q$, then $\tdel(0,\bW)^{-1} \tdel(\bZ,\bW)$ is bounded as $\bW$ approaches a point on the boundary.
\end{cor}
\begin{proof}
Note that
\[
\tdel(0,\bW)^{-1} \tdel(\bZ,\bW) = \\ \sqrt{\frac{1 - \|Q(\bW)\|^2}{\|Q(\bW)\|^2}} \left\| D_{Q(\bZ)}^{-1/2} (Q(\bZ) - Q(\bW)) D_{Q(\bW)^*}^{-1/2} \right\|.
\]
Note that $\|Q(\bW)\| \to 1$, as $\bW$ approaches the boundary of $\D_Q$. By the preceding lemma applied to $T = Q(\bW)^* Q(\bW)$,
\begin{multline*}
D_{Q(\bZ)}^{-1/2} (Q(\bZ) - Q(\bW))(1 - \|Q(\bW)\|^2)D_{Q(\bW)^*} (Q(\bZ)^* - Q(\bW)^*) D_{Q(\bZ)}^{-1/2} \leq \\ D_{Q(\bZ)}^{-1/2} (Q(\bZ) - Q(\bW)) (Q(\bZ)^* - Q(\bW)^*) D_{Q(\bZ)}^{-1/2}.
\end{multline*}
The right-hand side is bounded in norm, and we are done.
\end{proof}

\begin{defn}
Let $\zeta \in \partial \D_Q(1)$. We denote by $w$ points in $\D_Q(1)$. The big and small horospheres of radius $R > 0$ at $\zeta$ are
\[
F_R(\zeta)(n) = \left\{ \bZ \in \D_Q \mid \liminf_{w \to \zeta} \tdel(0,w )^{-1} \tdel(\bZ,w \otimes I_n) < R \right\},
\]
\[
E_R(\zeta)(n) = \left\{ \bZ \in \D_Q \mid \limsup_{w \to \zeta} \tdel(0,w)^{-1} \tdel(\bZ,w \otimes I_n) < R \right\}.
\]
\end{defn}
Note that since $\D_Q$ is unitary invariant and closed under direct sums, for every $\bW \in \D_Q$, $\bW \otimes I_n \in \D_Q$, since it is up to a canonical shuffle $I_n \otimes \bW = \bW^{\oplus n} \in D_Q$. Moreover, since $Q$ is nc we have that $Q(\bW \otimes I_n) = Q(\bW) \otimes I_n$. For a point $w \in \D_Q(1)$, we have $w \otimes I_n = I_n \otimes w$.

\begin{lem} \label{lem:small_horosphere_convex}
If $\D_Q$ is matrix convex, then for every $R > 0$ and $\zeta \in \partial \D_Q(1)$, $E_R(\zeta)$ is matrix convex.
\end{lem}
\begin{proof}
Let $\bZ \in E_R(\zeta)(n)$ and $V \colon \C^m \to \C^n$ an isometry. Then, for $w \in \D_Q(1)$, $V^* (w \otimes I_n) V = w \otimes I_m$ and thus, by Lemma \ref{lem:delta_and_isom}, we have that $\tdel(V^*\bZ V, w \otimes I_m) \leq \tdel(\bZ, w \otimes I_n)$ and thus $V^*\bZ V \in E_R(\zeta)(n)$. Moreover, if $\bZ_1 \in E_R(\zeta)(n)$ and $\bZ_2 \in E_R(\zeta)(m)$, then 
\[
\tdel(\bZ_1 \oplus \bZ_2, w \otimes I_{n+m}) = \max\left\{ \tdel(\bZ_1, w\otimes I_n), \tdel(\bZ_2, w\otimes I_m)\right\}.
\]
Hence, $E_R(\zeta)$ is closed under direct sums.
\end{proof}

The following lemma lists some properties of the horospheres. These properties follow immediately from the definition, and we include them for the sake of completeness (see  \cite{Abate-horospheres, AbateRaissy}).

\begin{lem} \label{lem:horospheres_properties}
Fix a domain bounded $\D_Q$.
\begin{enumerate}
    \item[(i)] For all $R > 0$, $E_R(\zeta) \subset F_R(\zeta)$.
    
    \item[(ii)] For $0 < R_1 < R_2$, $\overline{E_{R_1}(\zeta)} \subset E_{R_2}(\zeta)$ and $\overline{F_{R_1}(\zeta)} \subset F_{R_2}(\zeta)$.
    
    \item[(ii)] $\D_Q = \bigcup_{R> 0} E_R(\zeta) = \bigcup_{R > 0} F_R(\zeta)$.
    
    \item[(iii)] $\bigcap_{R > 0} E_R(\zeta) = \bigcap_{R > 0} F_R(\zeta) = \emptyset$.
    
\end{enumerate}
\end{lem}

\begin{example}
In the classical case of the ball, we recover the classical definition of a horosphere in the ball. Indeed, if $\zeta = \left(\begin{smallmatrix} 1 & 0 & \ldots & 0 \end{smallmatrix} \right)$. In the case of the ball $Q(\bW)$ is the operator $\bW$, hence for $w \in \BB_d$, we have $Q(w) Q(w)^* = \|w\|^2$ and $Q(w)^* Q(w)$ is a rank one operator in $M_d$.

For $z, w \in \BB_d$, 
\[
\tdel(0,w)^{-2} \tdel(z,w)^2 = \frac{1 - \|w\|^2}{\|w\|^2 (1- \|z\|^2)} (Q(z) - Q(w))(1 - Q(w)^* Q(w))^{-1} ( Q(z) - Q(w))^*
\]
Write $v = Q(w)^*$, a column vector, and let $P_v$ be the orthogonal projection on the subspace spanned by $v$. Note that $P_v = \frac{1}{\|v\|^2} v v^*$. Therefore,
\begin{multline*}
(1 - Q(w)^* Q(w))^{-1} = (1 - vv^*)^{-1} = (1 - \|v\|^2 P_v)^{-1} = \\ \sum_{n=0}^{\infty} \|v\|^{2n} P_v^n = I + \frac{\|v\|^2}{1-\|v\|^2} P_v.
\end{multline*}
Now multiplying by $1 - \|Q(w)\|^2$ we get that
\[
(1 - \|Q(w)\|^2)(1 - Q(w)^* Q(w))^{-1} = (1 - \|Q(w)\|^2) + \|Q(w)\|^2 P_{Q(w)^*}.
\]
Thus, as $w \to \zeta$, we get that $\lim_{w \to \zeta} (1 - \|Q(w)\|^2)(1 - Q(w)^* Q(w))^{-1} = P_{Q(\zeta)^*}$. Therefore,
\[
\lim_{w \to \zeta} \tdel(0,w)^{-2} \tdel(z,w)^2 = \frac{|1 - z_1|^2}{1 - \|z\|^2}.
\]
Therefore, $E_R(\zeta)(1) = F_R(\zeta)(1)$ in this case and are the classical horospheres in the ball. Moreover, note that since the row $Q(w)$ is scalar, we have that $Q(w) \otimes I_n = I_n \otimes Q(w)$. Hence we have
\begin{multline*}
(1 - \|w\|)^2 (Q(\bZ) - Q(w) \otimes I_n) (1 - Q(w)^* Q(w) \otimes I_n)^{-1} (Q(\bZ) - Q(w) \otimes I_n)^* \stackrel{w \to \zeta}{\longrightarrow} \\ (Q(\bZ) - Q(\zeta) \otimes I_n) P_{\zeta} (Q(\bZ) - Q(\zeta) \otimes I_n)^* = (I - Z_1)(I-Z_1)^*.
\end{multline*}
Hence, 
\[
\lim_{w \to \zeta} \tdel(0,w)^{-2} \tdel(\bZ,w\otimes I_n)^2 = \|D_{\bZ}^{-1/2} (I - Z_1)\|.
\]
In particular, the limit exists. Thus, $E_R(\zeta) = F_R(\zeta)$. Moreover, for $\bZ \in E_r(\zeta)$
\[
(I - Z_1)(I-Z_1)^* < R^2 (I - \bZ \bZ^*).
\]
Thus, we recover Popescu's definition \cite{Popescu-composition} of a horosphere in the free ball.
\end{example}

\begin{lem} \label{lem:triangle}
For three distinct points $\bX,\bZ, \bW \in D_Q(n)$. Then,
\[
\delta(\bX,\bW)^{-1} \delta(\bZ,\bW) \leq   \dfrac{\delta(\bZ,\bX)}{\|Q(\bW)\| - \|Q(\bX)\|} + \|D_{Q(\bZ)^*}^{-1/2} D_{Q(\bX)^*}^{1/2}\|.
\]
\end{lem}
\begin{proof}
By the formula in Lemma \ref{lem:D_Q_nc_lempert}
\begin{multline*}
\delta(\bZ,\bW) = \| D_{Q(\bZ)}^{-1/2} (Q(\bZ) - Q(\bW)) D_{Q(\bW)^*}^{-1/2}\| \leq \\ \|D_{Q(\bZ)}^{-1/2} (Q(\bZ) - Q(\bX))D_{Q(\bW)^*}^{-1/2}\| + \|D_{Q(\bZ)}^{-1/2} (Q(\bX) - Q(\bW))D_{Q(\bW)^*}^{-1/2}\| \leq \\ \delta(\bZ,\bX) \|D_{Q(\bX)^*}^{1/2} D_{Q(\bW)^*}^{-1/2}\| + \delta(\bX,\bW) \|D_{Q(\bZ)^*}^{-1/2} D_{Q(\bX)^*}^{1/2}\|.
\end{multline*}
Since the points are distinct $\delta(\bX,\bW) \neq 0$. Moreover, since $Q(\bX)$ is a contraction, we have that $\|D_{Q(\bX)^*}^{1/2}\| \leq 1$. Hence, we obtain the following inequality:
\[
\delta(\bX,\bW)^{-1} \delta(\bZ,\bW) \leq \delta(\bZ,\bX) \frac{\| D_{Q(\bW)^*}^{-1/2}\|}{\delta(\bX,\bW)} + \|D_{Q(\bZ)^*}^{-1/2} D_{Q(\bX)^*}^{1/2}\|.
\]
We need to bound $\frac{\| D_{Q(\bW)^*}^{-1/2}\|}{\delta(\bX,\bW)}$. By functional calculus, $\| D_{Q(\bW)^*}^{-1/2}\| = \frac{1}{\sqrt{1 - \|Q(\bW)\|^2}}$. Hence, we will bound $\sqrt{1 - \|Q(\bW)\|^2} \delta(\bX,\bW)$ from bellow. Again, since $Q(\bX)$ is a strict contraction, $(1-Q(\bX)Q(\bX)^*)^{-1} \geq 1$. Let us write $$T = D_{Q(\bX)}^{-1/2} (Q(\bX) - Q(\bW)) D_{Q(\bW)^*}^{-1/2},$$  Then,
\[
T^* T \geq D_{Q(\bW)^*}^{-1/2} (Q(\bX)^* - Q(\bW)^*)(Q(\bX) - Q(\bW)) D_{Q(\bW)^*}^{-1/2}. 
\]
Hence,
\begin{multline*}
\sqrt{1 - \|Q(\bW)\|^2} \delta(\bX,\bW) = \sqrt{1 - \|Q(\bW)\|^2} \|T\| \geq  \\ \sqrt{1 - \|Q(\bW)\|^2} \|(Q(\bX) - Q(\bW)) D_{Q(\bW)^*}^{-1/2}\| \geq \\ \sqrt{1 - \|Q(\bW)\|^2} \left( \|Q(\bW) D_{Q(\bW)^*}^{-1/2}\| - \|Q(\bX) D_{Q(\bW)^*}^{-1/2}\| \right)
\end{multline*}
Now, if $A = Q(\bW) D_{Q(\bW)^*}^{-1/2}$, then $A^* A = Q(\bW)^* Q(\bW) D_{Q(\bW)^*}^{-1}$. Thus, by functional calculus, $\|A\| =\frac{\|Q(\bW)\|}{\sqrt{1 - \|Q(\bW)\|^2}}$. Now, for the second summand, we have
\[
\sqrt{1 - \|Q(\bW)\|^2} \|Q(\bX) D_{Q(\bW)^*}^{-1/2}\| \leq \\ \sqrt{1 - \|Q(\bW)\|^2} \|Q(\bX)\| \|D_{Q(\bW)^*}^{-1/2}\| = \|Q(\bX)\|.
\]
In summary,
\[
\sqrt{1 - \|Q(\bW)\|^2} \delta(\bX,\bW) \geq \|Q(\bW)\| - \|Q(\bX)\|.
\]
Thus,
\[
\delta(\bX,\bW)^{-1} \delta(\bZ,\bW) \leq  \dfrac{\delta(\bZ,\bW)}{\|Q(\bW)\| - \|Q(\bX)\|} + \|D_{Q(\bZ)^*}^{-1/2} D_{Q(\bX)^*}^{1/2}\|.
\]
\end{proof}

\begin{thm}[Noncommutative Wolff theorem] \label{thm:nc_wolff_lemma}
Let $0 \in \Omega = \D_Q$ be a bounded matrix convex set. Let $f \colon \Omega \to \Omega$ be an nc map with no fixed points. Then there exists $\xi \in \partial \Omega(1)$, such that for every $R > 0$ and $n \in \N$, $f^{\circ n}(E_R(\xi)) \subseteq F_R(\xi)$.
\end{thm}
\begin{proof}
Let $0 < r_m < 1$ be a sequence that monotonically increases to $1$. Let $f_m(\bZ) = f(r_m\bZ)$. By Lemma \ref{lem:fixed_point_scalar}, $f_m$ has a unique fixed point $w_m \in \Omega(1)$. Any accumulation point of $w_m$ in $\Omega(1)$ will be a fixed point of $f$. Hence (passing to a subsequence, if necessary), we may assume that $w_m \to \xi \in \partial \Omega$. Fix $R > 0$, and let $\bZ \in E_R(\xi)(n)$. Then
\[
\limsup_{m\to \infty} \tdel(0,w_m)^{-1} \tdel(\bZ,w_m \otimes I_n) \leq \limsup_{w \to \xi} \tdel(0,w)^{-1} \tdel(\bZ,w \otimes I_n) < R.
\]
Fix an arbitrary $\varepsilon > 0$ small enough. Then, for every $m$ big enough,
\[
\tdel(0,w_m)^{-1} \tdel(\bZ,w_m \otimes I_n) < R - \varepsilon.
\]
By \cite[Corollary 3.4]{BelVin17} and the fact that $w_m$ is a fixed point of $f_m^{\circ n}$, for every $n \in \N$, we conclude that for every $m > m_0$,
\[
\tdel(f_m^{\circ n}(\bZ), f_m^{\circ n}(w_m) \otimes I_n) = \tdel(f_m^{\circ n}(\bZ), w_m \otimes I_n) \leq \tdel(\bZ,w_m\otimes I_n) <  \tdel(0,w_m)(R - \varepsilon).
\]

Now, fix $n \in \N$. For every $\bX \in \Omega$, $\lim_{m\to\infty} f^{\circ n}_m(\bX) = f^{\circ n}(\bX)$. By the continuity of $\tdel$, we have that $\lim_{m\to\infty} \tdel(f^{\circ n}_m(\bX), f^{\circ n}(\bX)) = 0$. By Lemma \ref{lem:triangle},
\begin{multline*}
\tdel(f^{\circ n}_m(\bZ), w_m \otimes I_n)^{-1} \tdel(f^{\circ n}(\bZ), w_m \otimes I_n) \leq \\  \frac{\tdel(f^{\circ n}(\bZ),f^{\circ n}_m(\bZ))}{\|Q(w_m)\| - \|Q(f^{\circ n}_m(\bZ))\|} + \|D_{Q(f^{\circ n}(\bZ))^*}^{-1/2} D_{Q(f^{\circ n}_m(\bZ))^*}^{1/2}\|.
\end{multline*}
Since $\lim_{m\to\infty} \frac{1}{\|Q(w_m)\| - \|Q(f^{\circ n}_m(\bZ))\|} = \frac{1}{1 - \|Q(f^{\circ n}(\bZ))\|}$, we get that
\[
\limsup_{m\to\infty} \tdel(f^{\circ n}_m(\bZ), w_m \otimes I_n)^{-1} \tdel(f^{\circ n}(\bZ), w_m \otimes I_n) \leq 1.
\]
In particular, for $m >> 0$,
\[
\tdel(f^{\circ n}(\bZ), w_m \otimes I_n) < (1 +\varepsilon) \tdel(f^{\circ n}_m(\bZ), w_m \otimes I_n) \leq (1 +\varepsilon) \tdel(\bZ, w_m\otimes I_n).
\]
Hence, if we take $0< \delta < \frac{\varepsilon}{R - \varepsilon}$, then again for every $m$ big enough,
\[
\tdel(0,w_m)^{-1} \tdel(f^{\circ n}(\bZ), w_m \otimes I_n) < (1 + \delta)(R - \varepsilon) < R.
\]
This implies that
\[
\liminf_{w \to \zeta} \tdel(0,w)^{-1} \tdel(f^{\circ n}(\bZ), w \otimes I_n) \leq \\ \liminf_{m \to \infty} \tdel(0,w_m)^{-1} \tdel(f^{\circ n}(\bZ), w_m \otimes I_n) < R.
\]
Thus, by definition, $f^{\circ n}(\bZ) \in F_R(\zeta)$.
\end{proof}

\bibliographystyle{abbrv}
\bibliography{nc_denjoy_wolff}
\end{document}